\documentclass[12pt]{amsart}
\usepackage{amsmath}
\usepackage{amsfonts}
\usepackage{amssymb}
\usepackage{graphicx} 
 
 \newtheorem{thm}{Theorem}[section]
 \newtheorem{lem}[thm]{Lemma}
 \newtheorem{exam}[thm]{Example}
 \newtheorem{prop}[thm]{Proposition}
 \newtheorem{cor}[thm]{Corollary}
 \newtheorem{rem}[thm]{Remark}
 \newtheorem{defn}[thm]{Definition}
\newcommand{\module}[1]{\left\vert #1 \right\vert}
\newcommand{\m}{\mathbf{m}}
 \def\Hol{\mathop{\rm Hol}\nolimits}
 \def\Aut{\mathop{\rm Aut}\nolimits}
 \def\Id{\mathop{\rm Id}\nolimits}
 
 \def\dist{\mathop{\rm dist}\nolimits}
 \def\Im{\mathop{\rm Im}\nolimits}
 \def\rg{\mathop{\rm rg}\nolimits}
 \def\span{\mathop{\rm Span}\nolimits}

\newcommand{\C}{{\mathbb C}}   
 \newcommand{\R}{{\mathbb R}}           
\newcommand{\N}{{\mathbb N}}        
\newcommand{\Z}{{\mathbb Z}}          
\newcommand{\LX}{{\mathcal L}(X)}     
\newcommand{\D}{{\mathbb D}}

\author{W. Arendt}  
\address{Wolfgang Arendt, Institute of Applied Analysis, University of Ulm. Helmholtzstr. 18, D-89069 Ulm (Germany)} 
\email{wolfgang.arendt@uni-ulm.de}
\author{B. C\'elari\`es} 
\address{Benjamin C\'elari\`es, Univ Lyon, \'Ecole centrale de Lyon, Universit\'e Claude Bernard Lyon 1, CNRS UMR 5208, Institut Camille Jordan, F-69134 Écully, France} 
\email{celaries@math.univ-lyon1.fr}
\author{I. Chalendar}
\address{Isabelle Chalendar,  Universit\'e Paris-Est, LAMA, (UMR 8050), UPEM, UPEC, CNRS, F-77454, Marne-la-Vallée (France)}
\email{isabelle.chalendar@u-pem.fr}
\title[Diagonalization of composition operators]{In Koenigs' footsteps:  diagonalization of composition operators}
\begin{document}	

\begin{abstract}   
Let $\varphi:\D\to\D$ be a holomorphic map with a fixed point $\alpha\in\D$ such that $0\leq |\varphi'(\alpha)|<1$. We show that the spectrum of the composition operator $C_\varphi$ on the Fr\'echet space $\Hol(\D)$
 is $\{0\}\cup \{  \varphi'(\alpha)^n:n=0,1,\cdots\}$ and its essential spectrum is reduced to $\{0\}$. This contrasts the situation where a restriction of $C_\varphi$ to Banach spaces such as $H^2(\D)$ is considered. 
 Our proofs are based on  explicit formulae for the spectral projections associated with the  point spectrum found by Koenigs.  Finally, as a byproduct, we obtain information on the spectrum for bounded composition operators induced by a Schr\"oder symbol on arbitrary Banach spaces of holomorphic functions.   	
\end{abstract}	
\maketitle
\tableofcontents
\section{Introduction}

Let $\varphi$  be a  holomorphic self-map of the open unit disc $\D$ and let $\Hol(\D)$ be the algebra of holomorphic functions on $\D$ which is a Fr\'echet space endowed with the topology of uniform convergence on every compact subsets of $\D$. 

Denote by $\Aut(\D)$ the group of  all automorphisms on $\D$. It is a well-known fact that such functions have the form $z\mapsto e^{i\theta}\frac{z-a}{1-\overline{a}z}$ where $a\in\D$ and $\theta\in\R$.

The  functional equation $f\circ \varphi=\lambda f$ where $\lambda\in\C$ is called  the
\emph{homogeneous Schr\"oder} equation. 

For those $\varphi$ which are not automorphisms of $\D$ and which admit a fixed point $\alpha\in\D$, the solution was found by G. Koenigs in 1884. Note that a fixed point in $\D$  is unique whenever it exists. 

By $\N_0$ we denote the set of all nonnegative integers and let $\N=\N_0\setminus \{0\}=\{1,2,\ldots\}$.  


\begin{thm}[Koenigs' theorem]\label{th:koenigs}
Let $\varphi$ be a holomorphic map on $\D$ such that $\varphi(\D)\subset\D$, $\varphi\not\in\Aut(\D)$ and assume that $\varphi$ has a fixed point $\alpha\in\D$ with $\lambda_1:=\varphi'(\alpha)$. Then the following holds:
\begin{itemize}
	\item If $\lambda_1=0$ the equation $f\circ\varphi=\lambda f$ has a nontrivial solution $f\in\Hol(\D)$ if and only if $\lambda=1$ and the constant functions are the only solutions. 
	\item If $\lambda_1\neq 0$, then:
     \begin{itemize}
	\item[(a)] the equation $f\circ\varphi=\lambda f$ has a nontrivial solution $f\in\Hol(\D)$ if and only if $\lambda\in\{\lambda_1^n:n\in\N_0\}$;  
	\item[(b)] there exists a unique function $\kappa\in\Hol(\D)$ satisfying \[\kappa\circ\varphi=\lambda_1\kappa  \mbox{ and }\kappa'(\alpha)=1;\]
	\item[(c)] for $n\in\N_0$ and $f\in\Hol(\D)$, $f\circ \varphi=\lambda_1^n f$ if and only if $f=c\kappa^n$ for some $c\in\C$. 
     \end{itemize}
\end{itemize}
\end{thm}
The case where $\varphi'(\alpha)\neq 0$ is the most interesting one. To be consistent with \cite{shap98}, we use the following terminology. 
\begin{defn}
	A \emph{Schr\"oder map} is a holomorphic function $\varphi$ satisfying the following conditions:
	\begin{center}
	 $\varphi(\D)\subset\D$, $\varphi\not\in\Aut(\D)$, $\exists\alpha\in\D$ such that $\varphi(\alpha)=\alpha$ and  $\varphi'(\alpha)\neq 0$. 
	 \end{center}
	  The function $\kappa$ associated to a Schr\"oder map in Theorem~\ref{th:koenigs} is called  the \emph{Koenigs' eigenfunction} of $\varphi$. 
\end{defn}
 As a consequence of the Schwarz lemma \cite{rudin}, a holomorphic self-map $\varphi$ of $\D$ with a fixed point $\alpha\in\D$ is a Schr\"oder map if and only if $0<|\varphi'(\alpha)|<1$.     
 Moreover, Koenigs' eigenfunction $\kappa$ is then obtained as the limit of   $\frac{\varphi_n}{\lambda^n }$ in $\Hol(\D)$ as $n\to\infty$, where $\varphi_n=\varphi\circ\cdots\circ\varphi$. \\
 

The aim of this paper is to study the  non homogeneous Schr\"oder equation  
\begin{equation}\label{eq:schroeder}
f\circ \varphi-\lambda f=g
\end{equation} where $\lambda\in\C$ and $g\in \Hol(\D)$ are given and $f\in\Hol(\D)$ the solution. 
 
 As in Koenigs' work, we consider the case where  $\varphi\not\in\Aut(\D)$ and $\varphi$ has  a fixed point $\alpha$ in $\D$.  
 
 
 The study of the homogenenous Schr\"oder equation can be reformulated from an operator theory  point of view in the following way: 
 consider the composition operator $C_\varphi:\Hol(\D)\to\Hol(\D)$ given by $C_\varphi(f)=f\circ \varphi$. We denote by $\sigma(C_\varphi)$ the spectrum  and by $\sigma_p(C_\varphi)$ the point spectrum of $C_\varphi$. 
 Thus  (\ref{eq:schroeder}) has a unique solution for all $g\in\Hol(\D)$ if and only if $\lambda\not\in\sigma(T)$. Moreover, Koenigs' theorem implies that $\sigma_p(C_\varphi)=\{  \lambda_n:n\in\N_0\}$.

 
 
 
  
  Our main result consists in finding "the spectral projections" associated with $\lambda_n=\lambda_1^n$. The difficulty is that these spectral projections are not defined since a priori we do not know that the $\lambda_n$ are isolated in the spectrum. 
  We define projections $P_n$ of rank $1$ such that $P_nC_\varphi=C_\varphi P_n=\lambda_n P_n$. 
  Using these "spectral" projections we then show that actually the spectrum of the composition operator $C_\varphi$ on $\Hol(\D)$  is given by 
  \[     \sigma(C_\varphi)  =\{\lambda_n:n\in\N_0\} \cup\{0\} .\]  
This looks very similar to spectral properties of compact operators. But we show that the operator $C_\varphi$ is compact on $\Hol(\D)$ only in very special situations. 

Nevertheless, our results show that the operator $C_\varphi$ on $\Hol(\D)$ is always a Riesz operator; i.e. its essential spectrum is reduced to $\{0\}$. This contrasts the situation where a restriction $T={C_\varphi}_{|H^2(\D)}$ is considered. Indeed, in this case, the essential spectrum is a disc with $r_e(T)>0$ in many cases. Actually, much is known on such restrictions to spaces such as $H^p(\D)$, Bergman space, Dirichlet space and others. See the monographs \cite{CM} of Cowen and MacCluer and of Shapiro \cite{shapiro}, as well as the  articles     \cite{CM94,hurst,Kam,KM,MS,MZZ,sha87,shapiroT,zheng} to name a few.

Our results on $\Hol(\D)$ allow us to prove some spectral properties of the restriction $T$ of $C_\varphi$ to some invariant Banach space $X\hookrightarrow \Hol(\D)$.  For instance we will see that $0\in\sigma(T)$ if and only if $\dim X=\infty$, and in this case we show that the essential spectrum $\sigma_e(T)$ is the connected component of $0$ in $\sigma_e(T)$. \\

 The paper is organized as follows. In Section~\ref{sec:2} we characterize composition operators on $\Hol(\D)$ as the non-zero algebra homomorphisms. This is also an interesting example of automatic continuity. We also characterize when $C_\varphi$ is compact as operators on $\Hol(\D)$ (which is much more restrictive than on $H^2(\D)$, for example). Section~\ref{sec:3} is devoted to the definition and investigation of the spectral projections. The main theorem determining the  spectrum of $C_\varphi$ in $\Hol(\D)$ is established in Section~\ref{sec:4}. Finally, we deduce spectral properties of restrictions to  arbitrary invariant Banach spaces in Section~\ref{sec:5}.      
  
  
 \section{Composition operators on $\Hol(\D)$}\label{sec:2}
  Let $\varphi$ be a holomorphic self-map of $\D$. We define the composition operator $C_\varphi$  on $\Hol(\D)$ by $C_\varphi(f)=f\circ \varphi$. Then $C_\varphi$ is in ${\mathcal L}(\Hol(\D))$ the algebra of linear and continuous operators on $\Hol(\D)$; indeed the linearity is trivial and  the continuity follows from the definition of the topology of the Fr\'echet space (uniform convergence on compact subsets of $\D$) and the continuity of $\varphi$. \\
  The next proposition is an algebraic characterization of composition operators. Note that $\Hol(\D)$ is an algebra. An algebra homomorphism $A:\Hol(\D)\to\Hol(\D)$ is a linear map satisfying 
  \[A(f\cdot g) =A(f)\cdot A(g) \mbox{ for all }f,g\in\Hol(\D.)  \]
  \begin{prop}\label{prop:compoalg}
  	Let  $A:\Hol(\D)\to \Hol(\D)$ be linear. 
  	The following assertions are equivalent.
  \begin{itemize}
  	\item[(i)] There exists a holomorphic map $\varphi:\D\to\D$ such that $A=C_\varphi$;
  	\item[(ii)] $A$ is  an  algebra homomorphism different from $0$;
  	\item[(iii)] $A$ is continuous and $Ae_n=(Ae_1)^n$ for all $n\in\N_0$.   
  \end{itemize}	
   \end{prop}
Here we define $e_n\in\Hol(\D)$ by $e_n(z)=z^n$ for all $z\in\D$ and all $n\in\N_0$. 
For the proof we use the following well-known result.   
\begin{lem}\label{lem:homo}
Let $L:\Hol(\D)\to\C$ be a continuous algebra homomorphism, $L\neq 0$. Then there exists $z_0\in\D$ such that $Lf=f(z_0)$ for all $f\in\Hol(\D)$.  
\end{lem} 
\begin{proof}
Since $Lf=L(f\cdot e_0)=L(f)L(e_0)$ for all 	$f\in\Hol(\D)$ and since $L\neq 0$, it follows that $L(e_0)=1$. Set $z_0:=Le_1$.
Then $z_0\in\D$. Indeed, otherwise $g(z)=\frac{1}{z-z_0}$ defines a function $g\in\Hol(\D)$ such that $(e_1-z_0e_0)g=e_0$. Hence  
\[ 1=Le_0=( Le_1-z_0Le_0)Lg=0,   \]
a contradiction.
 For $f\in\Hol(\D)$ such that $f(z_0)=0$, we have $Lf=0$. Indeed, since there exists $g\in\Hol(\D)$ such that $f=(e_1-z_0e_0)g$, it follows  that  $L(f)=(L(e_1)-z_0L(e_0))L(g)=0$. \\
For an arbitrary $f\in\Hol(\D)$, note that $h:= f-f(z_0)e_0$ satisfies $h(z_0)=0$.   Hence $0=L(h)=L(f)-f(z_0)$. 
\end{proof}	
\begin{rem}
	We are grateful to H.G. Dales and J. Esterle for helping us with Lemma~\ref{lem:homo}. For much more information about automatic continuity, we refer to the monograph of Dales \cite{dales} and 
	 the  survey article of Esterle \cite{esterle}.   
\end{rem}
\begin{proof}[Proof of Proposition~\ref{prop:compoalg}]
$(ii)\Rightarrow (i)$: since $A\neq 0$, it follows as in Lemma~\ref{lem:homo} that $Ae_0=e_0$.  Let $z\in\D$. Then $L(f):=(Af)(z)$ is an algebra homomorphism and $L(e_0)=1$. By Lemma~\ref{lem:homo}, there exists $\varphi(z)\in\D$ such that $(Af)(z)=f(\varphi(z))$ for all $f\in\Hol(\D)$. In particular $\varphi=Ae_1\in\Hol(\D)$. \\
$(iii)\Rightarrow (ii)$: it follows from $(iii)$ that $A(fg)=A(f)A(g)$ if $f$ and $g$ are polynomials. Since the set of polynomials is dense in $\Hol(\D)$ and since the multiplication is continuous, $(ii)$ follows.\\
$(i)\Rightarrow (iii)$ is trivial.  	
\end{proof}

For our purposes, the following corollary is useful.  	
\begin{cor}\label{cor:aut}
Let $X=\Hol(\D)$ and $\varphi$ a holomorphic self-map of $\D$. The following assertions are equivalent:
\begin{itemize}
	\item[(i)] $C_\varphi$ is invertible in ${\mathcal L}(\Hol(D))$;
	\item[(ii)] $\varphi$ is an automorphism of $\D$.   
\end{itemize} 
\end{cor}   
\begin{proof}
$(ii)\Rightarrow (i)$ is clear since $C_{\varphi^{-1}}C_\varphi=C_\varphi C_{\varphi^{-1}}=\Id$, where $\Id$ denotes  the identity map on $X$. \\
$(i)\Rightarrow (ii)$: let $C_\varphi$ be invertible, $A=C_\varphi^{-1}$. Then $A$ is  an algebra homomorphism. By Proposition~\ref{prop:compoalg} there exists a holomorphic map $\psi:\D\to\D$ such that $A=C_\psi$. Then \[e_1=C_\varphi(C_\psi e_1)=\psi\circ\varphi\mbox{ and }e_1=C_\psi(C_\varphi e_1)=\varphi\circ\psi.\]
Thus $\varphi$ is an automorphism and $\psi=\varphi^{-1}$. 
\end{proof}	
 Next we want to characterize those $\varphi$ for which $C_\varphi$ is compact on $\Hol(\D)$. The reason of this investigation is the following. 
  One of our main points in the article is to show that the spectral properties of a composition operator $C_\varphi$ for $\varphi:\D\to\D$ with interior fixed point looks very much to what one knows from compact operators. However, as we will show now, for composition operators on $\Hol(\D)$, compactness is a very restrictive condition. 
 
 Recall that ${\mathcal V}\subset \Hol(\D)$ is a neighborhood of $0$ if and only if there exist a compact subset $K\subset \D$ and $\varepsilon>0$ such that 
 \[  {\mathcal V}_{K,\varepsilon}:=\{f\in\Hol(\D) :|f(z)|<\varepsilon \mbox{ for  all }z\in K \}\subset{\mathcal V}.   \]      
 A linear mapping $T:X\to X$ where $X$ is a Fr\'echet space, is called \emph{compact} if  there exists a neighborhood ${\mathcal V}$ of $0$ such that $T  {\mathcal V}$ is relatively compact. Each compact linear mapping is continuous. We refer to Kelley--Namioka \cite{KN} for these notions and properties of compact operators. 
  
\begin{thm}\label{th:compact}
Let $\varphi:\D\to\D$ be holomorphic. The following assertions are equivalent:
\begin{itemize}
	\item[(i)] $C_\varphi$ is compact as operator from $\Hol(\D)$ to $\Hol(\D)$;
	\item[(ii)] $\sup_{z\in\D}|\varphi(z)|<1$.
\end{itemize} 	
\end{thm}
\begin{proof}
$(i)\Rightarrow (ii)$: assume that $\varphi(\D)\not\subset r\overline{\D}$ for all $0<r<1$. Let ${\mathcal V}$ be a neighborhood of $0$. We show that $C_\varphi ({\mathcal V})$ is not relatively compact. There exists $0<\varepsilon<1$ and $0<r_0<1$ such that 
\[  {\mathcal V}_0:=\{  f\in \Hol(\D):|f(z)|<\varepsilon, \, \forall z\in r_0\overline{\D} \}\subset {\mathcal V}.     \]
Thus it is sufficient to show that $C_\varphi ( {\mathcal V}_0)$ is not relatively compact. 
By our assumption there exists $w_0\in\D$ such that $z_0:=\varphi(w_0)\not\in r_0\overline{\D}$. Then there exist $r_0<r_1<1$ and $\rho>0$ such that $r_1\overline{\D}\cap D(z_0,\rho)=\emptyset$. 

The set $K:=r_0\overline{\D}\cup \{ z_0 \}$ is compact and $\C\setminus K$ is connected. Let $n\in\N_0$ and define $h_n$  by 
\[h_n(z)=0 \mbox{ for } z\in r_1\D \mbox{ and }h_n(z)=n+1\mbox{ for }z\in D(z_0,\rho).\]
 Set $\Omega:=r_1\D \cup D(z_0,\rho)$. 
Then $K\subset \Omega$ and $h_n:\Omega\to \C$ is holomorphic. By Runge's theorem, there exists a polynomial $p_n:\C\to\C$ such that $|p_n(z)-h_n(z)|<\varepsilon$ for all $z\in K$. 
This implies that ${p_n}_{| \D}\in {\mathcal V}_0$ and $|p_n(z_0)|\geq n$. 
Since $|C_\varphi (p_n)(w_0)|=|p_n(z_0)|\geq n$, the sequence $(C_\varphi p_n)_{n\in\N_0}$ has no convergent subsequence.\\
$(ii)\Rightarrow (i)$: Assume that $\sup_{z\in\D}|\varphi(z)|=:r_0<1$. The set 
\[  {\mathcal V}:=\{   f\in \Hol(\D)   :|f(z)|<1\mbox{ if }|z|\leq r_0  \}  \]
is a neighborhood of $0$. Let $f\in{\mathcal V}$. Since $\varphi(\D)\subset r_0(\overline{\D})$, one has $|f(\varphi(w))|<1$  for all $w\in\D$. Now it follows from Montel's theorem that $C_\varphi {\mathcal V}$ is relatively compact in $\Hol(\D)$.       	
\end{proof}	

\begin{rem}
The same characterization of compact composition operators is valid in some special Banach spaces of holomorphic functions, for example $X=H^\infty(\D)$ \cite{schwartz}.  However on $H^2(\D)$, the class of mappings $\varphi$ such that $C_\varphi$ is compact is much larger \cite{CM}.   
\end{rem}

\section{Diagonalization of composition operators}\label{sec:3}
In this section we show that composition operators $C_\varphi$ on $\Hol(\D)$ can be diagonalized if the symbol $\varphi$ is a Schr\"oder map. 

For the following we fix the holomorphic function $\varphi:\D\to\D$, with interior fixed point $\alpha=\varphi(\alpha)\in\D$ and suppose that $\varphi\not\in\Aut(\D)$, $\varphi'(\alpha)\neq 0$. We set $\lambda_n=\varphi'(\alpha)^n$ for $n\in\N_0$. Thus $\lambda_1\in\D$ by the Schwarz lemma and $|\lambda_n|$ tends to $0$ as $n\to\infty$. The range of an operator $T$ is denoted by $\rg T$. We denote by $\kappa$ Koenigs' eigenfunction associated with $\varphi$. The following properties of $\kappa^n$ will be needed. 
\begin{lem}\label{lem:deriv_sigma}
	For all $n \in \N$, $(\kappa^n)^{(n)}(\alpha) = n!$ and $(\kappa^n)^{(l)}(\alpha) =0$  for $l=0,\cdots, n-1$. 
\end{lem}
\begin{proof}
	%
	%
Since $\kappa (\alpha) = 0$ and $\kappa^\prime (\alpha)=1$, we get that, as $z\to\alpha$, 
	\[
		\kappa^n(z)  =  \left[ (z-\alpha) + o(z - \alpha)\right]^n 
		 =  (z - \alpha)^n + o((z - \alpha)^n).
	\]
	Hence,  $(\kappa^n)^{(n)}(\alpha) = n!$ and $(\kappa^n)^{(l)}(\alpha) =0$  for $l=0,\cdots, n-1$. 
\end{proof}
In the following theorem we define inductively a series of rank-one projections which diagonalize the operator $C_\varphi$  on $\Hol(\D)$.  
\begin{thm}\label{th:31}
	Define iteratively rank-one projections $P_n\in{\mathcal L}(\Hol(\D))$  by 
	\begin{equation}\label{eq:proj}
	P_0f=f(\alpha) e_0 \mbox{ and  for $n\in\N$, }P_n(f)=\frac{1}{n!}g^{(n)}(\alpha)\kappa^n,   
	\end{equation}
	where $g=f-\sum_{k=0}^{n-1}P_kf$.  Then the following holds:
	\begin{itemize}
		\item[(a)] $P_nC_\varphi =C_\varphi P_n=\lambda_n P_n$. 
		\item[(b)]  $f^{(l)}(\alpha)=\left( \sum_{k=0}^n P_kf\right)^{(l)}(\alpha)$ for $l=0,\cdots, n$ and $f\in\Hol(\D)$.
		\item[(c)] There exist complex numbers $c_{n,m}$ ($n,m\in\N_0$) such that 
		\[P_nf=\left( \sum_{m=0}^n c_{n,m}f^{(m)}(\alpha)\right) \kappa^n .\]
		\item[(d)] $P_nP_m=\delta_{n,m}P_n$ for all $n,m\in\N_0$. 
	\end{itemize}
\end{thm}

We deduce the following decomposition property from Theorem~\ref{th:31}. Let 
\[\Hol_n(\alpha):=\{  f\in\Hol(\D):f(\alpha)=f'(\alpha)=\cdots =f^{(n)}(\alpha)=0\},\]
and $Q_n=\sum_{k=0}^nP_k$, where $P_k$ is given in Theorem~\ref{th:31}. 
\begin{cor}\label{cor:33}
The mappings $Q_n$ are projections commuting with $C_\varphi$. Moreover $\{ \kappa^l:l=0,\cdots ,n\}$ is a basis of $\rg Q_n$  and $\ker(Q_n)=\Hol_n(\alpha)$.  Thus we have the decomposition 
\[  \Hol(\D)=\span \{  \kappa^m:m=0,\cdots,n\}\oplus \Hol_n(\alpha)  \]
into two subspaces which are invariant by $C_\varphi$. 
\end{cor}	
As a consequence, ${C_\varphi}_{|\rg Q_n}$ is a diagonal operator since $C_\varphi (\kappa^l)=\lambda_l\kappa^l$ for $l=0,\cdots,n$. Of course, by definition $\kappa^0$ is the constant function equal to $1$. 
\begin{proof}[Proof of Corollary~\ref{cor:33}]
a) Let $g=\sum_{m=0}^n a_m\kappa^m\in \Hol_n(\alpha)$ where $a_m\in\C$. Then by
 Lemma~\ref{lem:deriv_sigma}, \[0=g(\alpha)=a_0, 0=g'(\alpha)=a_1,\cdots, 0=g^{(n)}(\alpha)=n!a_n.\]
 This shows that the functions $\kappa^m, m=0,\cdots,n$ are linearly independent and that 
 \[\span\{  \kappa^m:m=0,\cdots,n\}\cap \Hol_n(\alpha)=\{0\}.  \]
 b) Let $f\in\Hol(\D)$. Then, by Theorem~\ref{th:31}, $f-Q_nf\in\Hol_n(\alpha)$. This shows that 
 \[ \Hol(\D)=\span \{  \kappa^m:m=0,\cdots,n\}\oplus \Hol_n(\alpha)        \]
 and that $Q_n$ is the projection onto the first space along this decomposition.  
	
\end{proof}	
\begin{proof}[Proof of Theorem~\ref{th:31}]
	We define $P_n$ by the iteration equation (\ref{eq:proj}). \\
	At first we show (b) inductively. For $n=0$ it is trivial. Let $n\geq 1$ and assume that (b) is true for  $n-1$. Let $f\in\Hol(\D)$ and $0\leq l<n$. Since $\kappa(\alpha)=0$, $(\kappa^n)^{(l)}(\alpha)=0$ for $l<n$ (by Lemma~\ref{lem:deriv_sigma}), it follows that 
	\[  \left( \sum_{k=0}^n P_kf \right) ^{(l)}(\alpha)= \left( \sum_{k=0}^{n-1} P_kf \right) ^{(l)}(\alpha)=f^{(l)}(\alpha), \]
	by the inductive hypothesis. For $l=n$, we have 
	\begin{eqnarray*}
		\left( \sum_{k=0}^n P_kf \right) ^{(n)}(\alpha) & =  &\left(\sum_{k=0}^{n-1} P_kf\right)^{(n)}(\alpha)
		\\
		&  & + \frac{1}{n!} \left( f-\sum_{k=0}^{n-1}P_kf\right)^{(n)}(\alpha) (\kappa^n)^{(n)}(\alpha)  \\
		& = & f^{(n)}(\alpha),
	\end{eqnarray*}
	since $(\kappa^n)^{(n)}(\alpha)=n!$  (see Lemma~\ref{lem:deriv_sigma}). 
	Thus (b) is proved. \\
	It is clear that $(P_nf)\circ \varphi =\lambda_nP_nf$ since $\kappa^n \circ \varphi =\lambda_n \kappa^n$. We show inductively that $P_n(f\circ \varphi)=\lambda_n P_nf$. For $n=0$ this is trivial.  Let $n\geq 1$ and  assume now that $P_{l}(f\circ\varphi) =\lambda_{l}P_{l}f$ for all $l\leq n-1$. Note that 
	\[P_n (f\circ \varphi) =\frac{1}{n!} (\tilde{g})^{(n)}(\alpha)\kappa^n,  \]
	where \[\tilde{g} =f\circ \varphi-\sum_{k=0}^{n-1}P_k(f\circ \varphi)=g\circ\varphi\]
	by the inductive hypothesis, where $g=f-\sum_{k=0}^{n-1}P_k f$. 
	It follows that $P_n(f\circ \varphi)=\frac{1}{n!}(g\circ\varphi)^{(n)}(\alpha) \kappa^n $. 	
	Now let us introduce some more notations in order to compute $(g\circ\varphi)^{(n)}(\alpha)$. 
	For $n \in \N$, let
	\[J_n = \left\lbrace \m = (m_1, \cdots , m_n) \in \N^n ~\vert~ m_1 + 2m_2 +  \cdots + n m_n = n  \right\rbrace  \]
	and
	\[ K_n = J_n \setminus \left\lbrace (n,0,\cdots, 0) \right\rbrace \]
	For $\m = (m_1, \cdots , m_n)\in J_n$, set
	$ \module{\m} = m_1 + \cdots + m_n$ and note that, for $\m \in J_n$, $\m \in K_n$ if and only if $\module{\m}< n$. For $\m \in J_n$, we also define the following coefficients
	\[ C_\m^n = \frac{n!}{m_1! \, m_2! \, \cdots m_n !} \prod\limits_{j=1}^n \left( \frac{\varphi^{(j)}(\alpha)}{j!} \right)^{m_j}  \]
	These coefficients are inspired by Fa\`a di Bruno's Formula: indeed, if $g \in \Hol(\D)$, then, for every $n \in \N$,
	\[ (g \circ \varphi)^{(n)}(\alpha) = \sum\limits_{\m \in J_n} C_\m^n \, g^{(\module{\m})}(\alpha) = \left( \varphi^\prime (\alpha) \right)^n \, g^{(n)}(\alpha) + \sum\limits_{\m \in K_n} C_\m^n \, g^{(\module{\m})}(\alpha)\]
	Since $g^{(|\m|)}(\alpha)=0$ by (b), we get 
	\[  (g\circ\varphi)^{(n)}(\alpha) =\varphi'(\alpha)^ng^{(n)}(\alpha) + \sum_{\m\in K_n} C_{\m}^ng^{(|\m|)}(\alpha)=\lambda_ng^{(n)}(\alpha). \]
	Thus $P_n(f\circ \varphi)=\lambda_nP_nf$ for all $n\in\N_0$, which implies that $P_nC_\varphi =C_\varphi P_n=\lambda_n P_n$ for all $n\in\N_0$. Thus (a) is proved.\\
	We show inductively that (c) holds for suitable coefficients. It is obvious for $n=0$ and assume that
	$P_k$ has the property for all $k\leq n-1$. Then 
	\begin{eqnarray*}
		P_n f & = & \frac{1}{n!} (f-\sum_{k=0}^{n-1}P_kf)^{(n)}(\alpha)\kappa^n\\
		& = & \frac{1}{n!}\left( f^{(n)}(\alpha)+\sum_{k=0}^{n-1}\left( \sum_{l=0}^kc_{k,l}f^{(l)} (\alpha)\right) (\kappa^k)^{(n)}(\alpha)\right)\kappa^n, 
	\end{eqnarray*} 
	which proves the claim for $n$. \\   
	In order to prove (d), note that by the properties defining the projections and proved previously, 
	for all $k,l\in\N_0$, we have:
	\[  \lambda_kP_l\kappa^k=P_l(\kappa^k\circ\varphi)=\lambda_lP_l\kappa^k.     \] 
	Since $\lambda_k\neq\lambda_l$ for $l\neq k$, it follows that $P_l\kappa^k=0$. Hence $P_lP_k=0$ for $k\neq l$. \\
	It 	remains to show that $P_n^2=P_n$, which is equivalent to check that $P_n\kappa^n=\kappa^n$. We can show this easily and inductively  since $P_k\kappa^n=0$ for $k<n$ and $(\kappa^n)^{(n)}(\alpha)=n!$.  
\end{proof}	
We can now give explicit expressions of $P_n$ for $n=0,1,2,3$.  
\begin{cor}For all $f\in\Hol(\D)$, we have:
\[\begin{array}{l}
	 P_0 f =  f(\alpha)1 \\
	 P_1(f)=  f'(\alpha)\kappa \\ 
	 P_2(f) =  \frac{1}{2}\left(  f''(\alpha)+\frac{\varphi''(\alpha)}{\lambda_2-\lambda_1}f'(\alpha)\right)\kappa^2 \\
	   P_3(f)=\frac{1}{3\!}\left(  f'''(\alpha) +\frac{3\varphi''(\alpha)}{\lambda_2-\lambda_1}f''(\alpha)  + \left(   \frac{\varphi'''(\alpha)}{\lambda_3-\lambda_1} +\frac{3(\varphi''(\alpha))^2}{ (\lambda_1-\lambda_2)(\lambda_1-\lambda_3)}             \right)   f'(\alpha)           \right)\kappa^3   
\end{array}\]
\end{cor}
A natural question concerns the density of $\span\{\kappa^n:n\in\N_0\}$ in $\Hol(\D)$. The following proposition gives the answer. 
\begin{prop}\label{prop:density}
	The space
	 $\span\{\kappa^n:n\in\N_0\}$ is dense in the Fr\'echet space $\Hol(\D)$ if and only if $\varphi$ is univalent. 
\end{prop}
\begin{proof}
The function $\varphi$ is univalent if and only if $\kappa$ is univalent (see \cite{shap98}). Thus univalence  of $\varphi$ is necessary for the density of  $\span\{\kappa^n:n\in\N_0\}$. Conversely, assume that $\kappa$ is univalent. Then  	$\Omega:=\kappa(\D)$ is a simply connected domain. It follows from Runge's theorem (see \cite[Chap.~13 $\S$~1 Section~2]{Re}) that the algebra ${\mathcal A}^{(\Omega)}$ of all polynomials on $\Omega$ is dense in $\Hol(\Omega)$. Composition by  $\kappa$ shows that  $\span\{\kappa^n:n\in\N_0\}$ is dense in $\Hol(\D)$. 
 \end{proof}
We consider two illustrations. 	
\begin{exam}\label{ex:1}
\begin{enumerate}
	\item[(a)] Consider the univalent Schr\"oder symbol $\varphi(z)=\frac{z}{2-z}$. The Koenigs eigenfunction is $\kappa(z)=\frac{z}{1-z}$ and $\Omega=\kappa(\D)=\{z\in\C:\Re (z)>-1/2\}$.
	\item[(b)] 	Let $\varphi(z)=z\frac{z+1/2}{1+z/2}$ which satisfies $\varphi(0)=0$ and $\varphi'(0)=1/2$. Since $\kappa\circ \varphi(z)=\kappa(z)/2$, it follows that $\kappa(0)=0=\kappa(-1/2)$, which obviously contradics the density of   $\span\{ \kappa^n:n\in\N_0\}$ in the Fr\'echet space $\Hol(\D)$.  
	\end{enumerate}
\end{exam}

\section{The spectrum of composition operators on $\Hol(\D)$}\label{sec:4}
In this section we determine the spectrum of $C_\varphi$  on the Fr\'echet space $\Hol(\D)$. We suppose throughout that $\varphi:\D\to\D$ is a holomorphic map, $\varphi\not\in \Aut(\D)$, with an interior fixed point $\varphi(\alpha)=\alpha\in\D$, and that 
$0<|\varphi'(\alpha)|<1$. 
The case where $\varphi'(\alpha)=0$ is treated at the very end of this section. 
We let $\lambda_n=\varphi'(\alpha)^n, n\in\N_0$. 
 By $\sigma(C_\varphi)$ (resp. $\sigma_p(C_\varphi)$) we denote the \emph{spectrum} (resp. \emph{point spectrum}) of $C_\varphi$, that is the set $\{\lambda\in\C:\lambda\Id -C_\varphi\mbox{ is not bijective}\}$ (resp. $\{\lambda\in\C:\lambda\Id -C_\varphi\mbox{ is not injective}\}$).
 Note that for $\lambda\not\in \sigma(C_\varphi)$, $(\lambda\Id-C_\varphi)^{-1}$ is a continuous linear operator on $\Hol(\D)$ (by the closed graph theorem). 

Since $\varphi\not\in \Aut(\D)$, we already know  that $0\in\sigma(C_\varphi)$, by Corollary~\ref{cor:aut}. Moreover, by Koenigs' theorem,   
\[\sigma_p(C_\varphi)=\{\lambda_n:n\in\N_0\}.\]
Now we show that the entire spectrum $\sigma(C_\varphi)$ is equal to  $\sigma_p (C_\varphi)\cup\{0\}$. This is surprising for several reasons. First of all, the operator  $C_\varphi$  is not compact in general (see Theorem~\ref{th:compact}). Nonetheless its spectral properties on $\Hol(\D)$ are exactly those of a compact operator (see  \cite{vasil} for the Riesz theory for compact operators on Fr\'echet spaces which is the same as for Banach spaces). The other surprise comes from the well developed spectral theory of ${C_\varphi}_{|X}$ for invariant Banach space $X\hookrightarrow\Hol(\D)$, which shows in particular that, on $X$, the spectrum is much larger in general (see Section~\ref{sec:5}).        
\begin{thm}\label{th:34}
One has 
 \[\sigma(C_\varphi)=\{0\}\cup \{\varphi'(\alpha)^n:n\in\N_0\}.\]    
\end{thm}
In order to prove the surjectivity of $C_\varphi -\lambda\Id$ on $\Hol(\D)$ for a complex number $\lambda\not\in  \{0\}\cup \{\varphi'(\alpha)^n:n\in\N_0\}$, we will use the following lemma.
\begin{lem}\label{lem:35}
	Let $\psi:\D\to \D$  be holomorphic, $\psi\not\in\Aut(\D)$, such that $\psi(0)=0$. Let  $g\in\Hol(\D)$ and $\lambda\in\C\setminus\{0\}$.  Assume that there exist $0<\varepsilon<1$ and $f\in\Hol(\varepsilon \D)$ such that 
	\[  \lambda f-f\circ \psi =g\mbox{ on }\varepsilon \D.  \]
	Then $f$ has an extension $\tilde{f}\in\Hol(\D)$ such that 
	\[  \lambda \tilde{f}- \tilde{f}\circ \psi =g \mbox{ on }\D.  \] 
\end{lem}
\begin{proof}
Let $\rho:=\sup\{  r\in [\varepsilon,1]:f\mbox{ has an analytic extension on }r\D \}$. We show that $\rho=1$. Assume that $\rho<1$. Then there exists $\tilde{f}\in\Hol(\rho\D)$, a holomorphic extension of $f$, satisfying:
 \[  \lambda \tilde{f}- \tilde{f}\circ \psi =g \mbox{ on }\varepsilon\D.  \] 	
 Since both sides are holomorphic, by the uniqueness theorem, the identity remains true on $\rho\D$. 
 Note that by the Schwarz lemma $\psi(r\D)\subset r\D$ for all $0<r<1$. 
 It follows also from the Schwarz lemma  that there exists $\rho<\rho'\leq 1$ such that $\psi(\rho'\D)\subset \rho\D$.  Indeed, otherwise we find $(z_n)_n\in\D$, $|z_n|\downarrow \rho$ such that $|\psi(z_n)|>\rho$. Taking a  subsequence we may assume that $z_n\to z$ and then $|z|=\rho$ and $|\psi(z)|\geq \rho$. This is not possible since $\psi$ is not an automorphism. Now, since 
 \[  \lambda \tilde{f}= \tilde{f}\circ \psi +g \mbox{ on }\rho \D,  \]
 and since $\psi(\rho'\D)\subset \rho \D$, it follows that $f$ has a holomorphic extension to $\rho'\D$, a contradiction to the choice of $\rho$.   
\end{proof}	
\begin{proof}[Proof of Theorem~\ref{th:34}]
\emph{First case}: $\alpha=0$. Let 	$\lambda\in\C$ and $\lambda\not\in \{0\}\cup \{\lambda_n:n\in\N_0 \}$. From Koenigs' theorem we know 
that $\lambda\Id -C_\varphi$ is injective. Thus we only have to prove the surjectivity. Let $g\in\Hol(\D)$ and choose $n\in\N_0$ such that $|\lambda_{n+1}|<|\lambda|$. Since by Corollary~\ref{cor:33}, $\Hol(\D)=\rg Q_n \oplus \Hol_n(0)$ we can write $g=g_1+g_2$ where $g_1\in \rg Q_n$ and $g_2\in\Hol_n(0)$. Since ${C_\varphi}_{|\rg Q_n}$ is a diagonal operator and $\lambda\not\in\sigma({C_\varphi}_{|\rg Q_n})$, there exists $f_1\in \rg Q_n$ such that $\lambda f_1-f_1\circ\varphi=g_1$. Next we look at $g_2$. Choose $|\lambda_1|<q<1$ such that $q^{n+1}<|\lambda|$. Since $\lim_{z\to 0}\frac{\varphi(z)}{z}=\lambda_1$, there exists $0<\varepsilon\leq 1$ such that $|\varphi(z)|\leq q|z|$ for $|z|<\varepsilon$. 
Consider the iterates $\varphi_k:=\varphi\circ\cdots\circ\varphi$  ($k$ times) of $\varphi$. Then $|\varphi_k(z)|\leq q^k|z|\leq q^k\varepsilon$ for $|z|<\varepsilon$.  Since $g_2\in\Hol_n(0)$, there exists $B\geq 0$ such that 
\[  |g_2(z)|\leq B |z|^{n+1}\mbox{ for }|z|<\varepsilon.  \]        
Hence for $k\in\N_0$, $|z|<\varepsilon$, 
 \begin{eqnarray*}
\left|  \frac{g_2(\varphi_k(z))}{\lambda^{k+1}}   \right|  & \leq & \frac{1}{|\lambda|}B\frac{|\varphi_k(z)|^{n+1}}{|\lambda|^k}\\
                                     & \leq & \frac{1}{|\lambda|}B \frac{q^{k(n+1)}}{|\lambda|^k}\varepsilon\\
                                     & \leq & \frac{B\varepsilon}{|\lambda|} \left(     \frac{q^{n+1}}{|\lambda|}  \right)^k.   	
\end{eqnarray*}  
Since $\frac{q^{n+1}}{|\lambda|}<1$, the series $f_0(z):=\sum_{k=0}^\infty \frac{g_2(\varphi_k(z))}{\lambda^{k+1}}$ converges uniformly on $\varepsilon \D$ and defines a function $f_0\in \Hol(\varepsilon\D)$. Moreover, since $\varphi(\varepsilon \D)\subset \varepsilon \D$, 
\begin{eqnarray*}
f_0(\varphi(z)) & = & \sum_{k=0}^\infty \frac{g_2(\varphi_{k+1}(z))}{\lambda^{k+1}} \\
 & = &  \lambda \sum_{k=1}^\infty \frac{g_2(\varphi_k (z))}{\lambda^{k+1}}\\
 & =   & \lambda f_0(z) -g_2(z) 	
\end{eqnarray*}
on $\varepsilon\D$. 
It follows from Lemma~\ref{lem:35} that $f_0$ has a holomorphic extension $f\in\Hol(\D)$ satisfying $\lambda f-f\circ \varphi =g_2$. This shows that $\lambda\not\in \sigma(C_\varphi)$ in the case $\alpha=0$. \\
\emph{Second case}: $\alpha\in\D$, $\alpha\neq 0$.  Consider the M\"obius transform $\psi_\alpha:\D\to\D$ defined by $\psi_\alpha(z)=\frac{\alpha-z}{1-\overline{\alpha}z}$ and note  that $\psi_\alpha(0)=\alpha$ and $\psi_\alpha=\psi_\alpha^{-1}$.  Then 
$\tilde{\varphi}:=\psi_\alpha\circ \varphi\circ \psi_\alpha$ maps $\D$ into $\D$ and satisfies $\tilde{\varphi}(0)=0$. Since \[C_{\tilde{\varphi}}=C_{\psi_\alpha}C_\varphi C_{\psi_\alpha}=  C_{\psi_\alpha}C_\varphi C_{\psi_\alpha}^{-1},\]
the operators $C_{\tilde{\varphi}}$ and $C_\varphi$ are similar. From the first case we deduce that  
\[ \sigma(C_\varphi)\setminus \{0\}=\sigma( C_{\tilde{\varphi}})\setminus \{0\}=\sigma_p (C_{\tilde{\varphi}})\setminus\{0\}  = \sigma_p (C_{{\varphi}})\setminus\{0\} =\{ \lambda_n:n\in\N_0 \}. \]
 \end{proof}

For later purposes we extract the following lemma from the previous proof. 
\begin{lem}\label{lem:4.3}
	Let $n\in\N_0$, $\lambda\in\C$ such that $|\lambda|>|\lambda_{n+1}|$. Then for each $g\in\Hol_n(\alpha)$, there exists a unique $f\in\Hol_n(\alpha)$ solving the inhomogeneous Schr\"oder equation $\lambda f-f\circ \varphi =g$.   
\end{lem} 
\begin{proof}
Since $\kappa^k \not\in\Hol_n(\alpha)$ for $k\in \{0,1,\cdots,n\}$, uniqueness follows from Koenigs' theorem. In order to prove existence, we  only have to show that there exists $f\in\Hol(\D)$ such that $\lambda f-f\circ \varphi=g$. Then, since $Q_n g =0$, $f_1=f-Q_n f\in\ker Q_n=\Hol_n(\alpha)$ satisfies $\lambda f_1-f_1\circ \varphi=g$ as well. 

In the case $\alpha=0$ the existence of $f$ follows from the proof of Theorem~\ref{th:34}. So let $\alpha\neq 0$. Consider the M\"obius transform $\psi_\alpha$ defined in the proof of Theorem~\ref{th:34} and let $h=g\circ \psi_\alpha$. Then $h\in\Hol_n(0)$. Indeed, $h(0)=g(\alpha)=0$.  Moreover, for $l\in\{1,\cdots,n\}$, using Fa\`a di Bruno's formula (see the proof of Theorem~\ref{th:31} for notations), we get:
\begin{eqnarray*}
 h^{(l)}(0) & = & (g\circ \psi)^{(l)}(\alpha)\\
   & = & \psi'_\alpha(\alpha)^l g^{(l)}(\alpha) + \sum_{\m\in K_n}C_{\m}^m g_2^{(|\m|)}(\alpha)\\
    & = & 0.    
\end{eqnarray*}
Consider $\tilde{\varphi}=\psi_\alpha\circ \varphi\circ \psi_\alpha$. Since $\tilde{\varphi}(0)=0$ we can apply the first case and find $\tilde{f}\in\Hol(\D)$ such that $\lambda \tilde{f}-\tilde{f}\circ \tilde{\varphi}=h$. 
Then $f:=\tilde{f}\circ \psi_\alpha\in\Hol(\D)$ and 
\[g=h\circ \psi_\alpha=\lambda f-\tilde{f}\circ \tilde{\varphi}\circ \psi_\alpha=\lambda f-f\circ\psi_\alpha\circ \tilde{\varphi}\circ \psi_\alpha=\lambda f-f\circ \varphi. \]	      
\end{proof}	
Our next aim is to show that each $P_n$ is the spectral projection associated with $\lambda_n$ for each $n\in\N$. We use the following definition.   
\begin{defn}\label{def:riesz}
	Let $Y$ be a Fr\'echet space and $S:Y\to Y$ linear and continuous. 
	\begin{enumerate}
		\item[(1)] A  number $\lambda\in\sigma(T)$ is called a \emph{Riesz point} if $\lambda$ is isolated and if   there exists a decomposition $Y=Y_1\oplus Y_2$ in closed subspaces which are invariant by $S$ such that:
		 \[
		\dim Y_1<\infty, \sigma(S_{|Y_1})=\{\lambda\} \mbox{ and }(\lambda\Id -S)_{|Y_2} \mbox{ is invertible.}
		\]
		It is not difficult to see that this decomposition is unique. The projection $P:Y\to Y$ onto $Y_1$ along this decomposition  is called the \emph{spectral projection} associated with the Riesz point $\lambda$.  
		\item[(2)] $\sigma_e(T):=\{ \lambda\in\sigma(T):\lambda\mbox{ is not a Riesz point}\}$ is the \emph{essential spectrum} and 
		$r_e(T)=\sup\{  |\lambda|:\lambda\in\sigma_e(T)\}$ is the \emph{essential spectral radius}.  
		\item[(3)] $T$ is a \emph{Riesz operator} if $r_e(T)=0$.	
	\end{enumerate}	
\end{defn}
\begin{rem}\label{rem:4.5}
\begin{enumerate}
	\item[(1)] If $X$ is a Banach space, the existence of the decomposition as in (1) of Definition~\ref{def:riesz} for $\lambda\in\sigma(T)$ implies that $\lambda$ is an isolated point since the set of all invertible operators is open in $\LX$. This  last property is no longer true if $X$ is a Fr\'echet space (see Example~\ref{ex:sp}).   
	\item[(2)] If $X$ is a Banach space, then an isolated point $\lambda\in\sigma(T)$ is a Riesz point if and only if $\lambda$ is a pole of the resolvent whose residuum $P$ has finite rank. In that case $P$ is the spectral projection. Note that 
	\[  P=\frac{1}{2i\pi} \int_{|\mu-\lambda|=\varepsilon}R(\mu, T)d\mu.   \]
	\item[(3)] See \cite{D}, in particular \cite[Theorem~3.19]{D} for other equivalent definitions of Riesz operators on Banach spaces.   
\end{enumerate}	
\end{rem}
Let $P_n$ be the rank-one projections defined in Theorem~\ref{th:31} where $n\in\N_0$. 
\begin{thm}\label{th:4.6}
The operator $C_\varphi$ on $\Hol(\D)$ is a Riesz operator. Moreover, for each $n\in\N_0$, the spectral projection associated with $\lambda_n$ is $P_n$. 	
\end{thm}
\begin{proof}
Let $n\in\N_0$. We show that $\lambda_n$ is a Riesz point with spectral projection $P_n$. Since $P_nC_\varphi=C_\varphi P_n=\lambda_n P_n$ and $\rg P_n=\C \kappa^n$, it follows that the decomposition 
\[   \Hol(\D)=\C \kappa^n \oplus \ker P_n   \]
is invariant under $C_\varphi$. Moreover, $\sigma( {C_\varphi}_{|\C \kappa^n}   ) =\{\lambda_n\}$. Thus it suffices to  show that $(\lambda_n\Id -C_\varphi)_{|\ker P_n}$ is bijective. Since $\kappa^n \not\in\ker P_n$ injectivity follows from Koenigs' theorem. In order to prove surjectivity, let $g\in\ker P_n$. Then, by Corollary~\ref{cor:33}, $g=g_1 +g_2$ where $g_1\in\span \{ \kappa^m:m=0,\cdots, n  \}=:Z$, $g_2\in\Hol_n(\alpha)$. Since $P_n g_1=P_n g-P_ng_2=0$ and since ${C_\varphi}_{|Z}$ is diagonal, there exists $f_1\in Z$ such that $\lambda_n f_1-f_1\circ \varphi =g_1$. 
Note that $|\lambda_n|>|\lambda_{n+1}|$. Thus it follows from Lemma~\ref{lem:4.3} that there exists $f_2\in\Hol(\D)$ such that $\lambda_n f_2-f_2\circ \varphi=g_2$. Therefore $f:=f_1+f_2$ solves $\lambda_n f-f\circ \varphi  =g$. This shows that $\lambda_n$ is a Riesz point and $P_n$ is the associated spectral projection.      	
\end{proof}	 	
	
We want to prove that a version of the formula in (2), Remark~\ref{rem:4.5}, remains true for the operator $C_\varphi$ on $\Hol(\D)$. 

At first we deduce from  \cite[Lemma 3.2]{vasil}  
that the following holds. 

\begin{lem}\label{lem:39}
	Let $z\in\D,f\in\Hol(\D)$. The functions  
	\[  \lambda\mapsto ((\lambda\Id -C_\varphi)^{-1}f)(z):\C\setminus \sigma(C_\varphi)\to \C  \]
	is holomorphic. 
\end{lem} 

This can also be  seen directly from our proof of Theorem~\ref{th:34}. 

The following contour formula for $P_n$ will be useful in Section~\ref{sec:5}.  	
\begin{lem}\label{lem:310}
	Let $n\in\N$, $\varepsilon>0$ such that $\lambda_k\not\in D(\lambda_n,2\varepsilon)$ for all $k\in\N\setminus\{n\}$. Then, for all $f\in \Hol(\D)$
	\[  \frac{1}{2i\pi}\int_{|\lambda-\lambda_n|=\varepsilon} ((\lambda\Id -C_\varphi)^{-1} f)(z)d\lambda =(P_nf)(z).   \] 
\end{lem} 
\begin{proof}
	Write $f=(\Id-P_n)f+P_nf$. The function 
	\[\lambda\mapsto ((\lambda\Id-C_\varphi)^{-1}(\Id-P_n)f)(z)\]
	is holomorphic on $D(\lambda_n,2\varepsilon)$ whereas  $(\lambda\Id -C_\varphi)^{-1}P_nf=\frac{1}{\lambda-\lambda_n}P_nf$. From this the claim follows. 	
\end{proof}		
The spectrum in a Fr\'echet space may be neither  closed nor bounded. Indeed, 
here is an example of a composition operator on $\Hol(\D)$.   
\begin{exam}\label{ex:sp}
	Let $r\in(0,1)$ and the automorphism  
	\[\psi(z)=\frac{z+r}{1+rz}.\] 
	By Corollary~\ref{cor:aut}, $0\not\in \sigma(C_\psi)$ but $\sigma_p(C_\psi)=\C\setminus\{0\}$ since for all $\lambda\in\C\setminus\{0\}$ we have 
	\[ g_\lambda \circ \psi = \left( \frac{1+r}{1-r} \right)^\lambda g_\lambda \mbox{ where }g_\lambda(z)=\left( \frac{1+z}{1-z}\right)^\lambda.   \] 
	So, for $\mu=se^{i\theta}$ with $s>0$ and $\theta\in\R$, $g_\lambda\circ \psi=\mu g_\lambda$ when
	\[\Re(\lambda)=\frac{\ln s}{\ln((1+r)/(1-r))} \mbox{ and }  
	\Im(\lambda)=\frac{\theta +2k\pi}{\ln((1+r)/(1-r))}, k\in\Z.\]  
\end{exam}
For this reason one defines a larger spectrum, the Waelbroeck spectrum $\sigma_w(T)$ in the following way (see \cite{vasil}). 

Let $T\in{\mathcal L}(\Hol(\D))$. Then 
\begin{eqnarray*}
	\C\setminus \sigma_w(T) & :=&  \{\lambda \in\C\setminus \sigma(T) : \mbox{ there exists a neighborhood $\mathcal V$ of $\lambda$ such }\\
	&  & \mbox{	that the family $((\lambda\Id-T)^{-1})_{\lambda\in{\mathcal V}}$ is bounded}  \}.    
\end{eqnarray*}
Here a subset $A$ of $\Hol(\D)$ is called bounded if 
\[\sup_{f\in A}\sup_{z\in K} |f(z)|<\infty\]
for all compact subsets $K$ of $\D$. 

From the proof of Theorem~\ref{th:34} one sees that, in our case, $\sigma(C_\varphi)=\sigma_w (C_\varphi)$. Now we can use Fr\'echet theory (see Theorem~3.11 in \cite{vasil}). It tells us in particular that for the isolated point $\lambda_n\in\sigma_w(C_\varphi)$, there exists a unique projection $R_n$ commuting with $C_\varphi$ such that
\[\sigma \left({C_\varphi}_{|\rg R_n} \right)=\{\lambda_n\}\mbox{ and }\sigma\left( {C_\varphi}_{|\ker R_n} \right) =\sigma(C_\varphi)\setminus \{\lambda_n\}.\]
It is clear that $R_n=P_n$. 
Anyhow, we needed to define them differently since a priori it is not clear at all that $\lambda_n$ is an isolated point. 
		
Finally, we determine the spectrum of composition operators in the case where the symbol is not Schr\"oder but has an interior fixed point.  	
\begin{thm}\label{th:311}
	Let $\varphi:\D\to\D$ be holomorphic, $\alpha\in\D$ such that $\varphi(\alpha)=\alpha$. Assume that $\varphi'(\alpha)=0$. Then 
	\[  \sigma(C_\varphi)=\sigma_w(C_\varphi)=\{0,1\}.  \] 
\end{thm}
\begin{proof}
	Since $\varphi'(\alpha)=0$, $\varphi\not\in\Aut(\D)$ and thus, $0\in\sigma(C_\varphi)$ by Corollary~\ref{cor:aut}. Since the constant functions are in the kernel of $C_\varphi-\Id$, $1\in\sigma_p(C_\varphi)\subset \sigma(C_\varphi)$.    
	By Theorem~\ref{th:koenigs}, for $\lambda\not\in\{0,1\}$, $C_\varphi -\lambda\Id$ is injective. It remains to check that it is also surjective. \\
	\emph{First case}: $\alpha=0$. Then  
	$\varphi_n(z)=z^{2^n}\tau_n(z)$ where $\tau_n$ is a holomorphic self-map of the unit disc ($\tau_n(\D)\subset\D$ follows from the Schwarz lemma). Let $g\in\Hol(\D)$, $g(z)=g(0) + g_1(z)$ where $g_1\in\Hol(\D)$ and $g_1(z)=zg_2(z)$ with $g_2\in\Hol(\D)$.
	Note that $(\lambda\Id -C_\varphi)(\frac{g(0)}{\lambda-1}{\bf 1})=g(0){\bf 1}$. Moreover,  the series 
	\[f(z):=\frac{1}{\lambda}\sum_{n\geq 0}\frac{g_1(\varphi_n(z))}{\lambda^n}=\frac{1}{\lambda}\sum_{n\geq 0}\frac{z^{2^n}\tau_n(z)g_2(\varphi_n(z))}{\lambda^n}\]
	uniformly converges on every compact $K\subset\D\cap \{|z|<\sqrt{|\lambda|}\}$ (since $|z^{2^n}|\leq |z^{2n}|$ for $z\in\D$).  Note also that $\lambda f-f\circ\varphi=g_1$ on such $K$. 
	The surjectivity of $\lambda\Id -C_\varphi$ follows from Lemma~\ref{lem:35}.\\
	\emph{Second case}: $\alpha\neq 0$. We proceed as in the proof of Theorem~\ref{th:34}.    
\end{proof}

\section{Spectral properties on arbitrary Banach spaces}\label{sec:5}
In this section we study spectral properties of composition operators on arbitrary Banach spaces which are continuously injected in $\Hol(\D)$. 

Throughout this section we assume that $\varphi:\D\to\D$ is holomorphic, $\varphi\not\in\Aut(\D)$ and that there exists $\alpha\in\D$ 
such that $\varphi(\alpha)=\alpha$ and $\varphi'(\alpha)\neq 0$; i.e. $\varphi$ is a Schr\"oder function.  By $\kappa$ we denote Koenigs' eigenfunction. 
Let $X$ be a Banach space such that $X\hookrightarrow \Hol(\D)$ (which means that $X$ is a subspace of $\Hol(\D)$ and the injection is continuous; see \cite{delhi1} for equivalent formulations).  Assume that $C_\varphi X\subset X$ and define $T:X\to X$ by  $T={C_\varphi}_{|X}$. Then $T\in \LX$ by the closed graph theorem.   

As before we will consider the spectral projections $P_n$ on $\Hol(\D)$ 
 and let  $\lambda_n=\varphi'(\alpha)^n$, $n\in\N_0$. By Theorem~\ref{th:31}, $P_nf=\langle \Psi_n, f\rangle \kappa^n$, where $\Psi_n$ is a functional given by 
\[  \langle \Psi_n ,f\rangle =\sum_{m=0}^n c_{nm}f^{(m)}(\alpha).\]
It follows from Theorem~\ref{th:31} (a) that, \[ \langle C_\varphi f,\Psi_n\rangle =\lambda_n \langle f,\Psi_n\rangle,   \] 
for all $f\in\Hol(\D)$. This implies that $T'{\Psi_n}_{|X}=\lambda_n{\Psi_n}_{|X}$.  Thus, if ${\Psi_n}_{|X}\neq 0$, then $\lambda_n\in \sigma_p(T')\subset \sigma(T)$. We note this as a first result. 

\begin{prop}\label{th:41}
	Let $n\in\N_0$. Assume that ${\Psi_n}_{|X}\neq 0$. Then 
 \[\lambda_n\in\sigma_p(T')\subset \sigma(T).\]
\end{prop}
The following corollary concerns all the  
 classical Banach spaces $X$ of holomorphic functions on the unit disc.    
\begin{cor}
Assume that $e_n\in X$ for all $n\in \N_0$. Then $\lambda_n\in \sigma(T')\subset \sigma(T)$ for all $n\in\N_0$. 	
\end{cor}
\begin{proof}
We know that $\Psi_n\neq 0$ for all $n\in \N_0$. Since the polynomials are dense in $\Hol(\D)$, it follows that ${\Psi_n}_{|X}\neq 0$.  	
\end{proof}	
It follows from the decomposition result, Corollary~\ref{cor:33}, that the $\Psi_n$ separate points in $\Hol(\D)$, i.e. for $f\in\Hol(\D)$, $\langle \Psi_n,f\rangle =0$ for all $n\in\N_0$ implies $f=0$. 
\begin{cor}
	If $X\neq \{0\}$, then $r(T)>0$, where $r(T)$ is the spectral radius of $T$. 
\end{cor}
\begin{proof}
Since the $\Psi_n$, $n\in\N_0$, separate the functions of $\Hol(\D)$, there exists $n\in\N_0$ such that ${\Psi_n}_{|X}\neq 0$. Hence $\lambda_n\in \sigma(T)$ by Proposition~\ref{th:41}. 	
\end{proof}	
We need the following characterization of the finite dimension also for further arguments.
\begin{prop}\label{prop:64}
	The following assertions are equivalent:
\begin{enumerate}
	\item[(i)] $0\not\in\sigma(T)$;
	\item[(ii)] for only finitely many $n\in\N_0$ one has ${\Psi_n}_{|X}\neq 0$;
	\item[(iii)] $\exists J\subset \N_0$ finite such that $X=\span\{  \kappa^l: l\in J \}$;
	\item[(iv)] $\dim X<\infty$. 
\end{enumerate}	
\end{prop} 
\begin{proof}
$(i)\Rightarrow (ii)$: Since $\lambda_n\to 0$ as $n\to\infty$, this follows from Proposition~\ref{th:41}.\\
 $(ii)\Rightarrow (iii)$: Let $J:=\{  n\in\N_0:{\Psi_n}_{|X}\neq 0\}$. It follows from Corollary~\ref{cor:33} that the $\Psi_n$, $n\in\N_0$, separate $\Hol(\D)$. Thus the mapping 
 \[  f\mapsto (\langle \Psi_n , f \rangle)_{n\in J}:X\to \C^d,   \]
 with $d=| J|$ is injective and linear. It follows that $\dim X\leq d$. 
 It follows from Proposition~\ref{th:41} that $\{\lambda_n:n\in J\}\subset \sigma_p(T)$. Since all $\lambda_n$ are different, it follows that $\dim X\geq d$. We have shown that $\dim X=d$ and $\sigma_p(T)=\{\lambda_n:n\in  J\}$. Now it follows from Koenigs' theorem that $X=\span\{\kappa^l : l\in J\}$. \\
 $(iii)\Rightarrow (iv)$ is trivial.\\
 $(iv)\Rightarrow (i)$: Since $\dim X<\infty$, by Koenigs' theorem,
 \[\sigma(T)=\sigma_p(T)\subset \{\lambda_n:n\in\N_0\}\subset \C\setminus \{0\}.\]    

\end{proof}

We note the following corollary which will be useful later.
\begin{cor}\label{cor:55}
Assume that $\dim X=\infty$. Then there exist infinitely many $n\in\N_0$ such that $\lambda_n\in\sigma(T)$.  
\end{cor}
This follows from Proposition~\ref{th:41} and Proposition~\ref{prop:64}. \\

Next we show that each isolated point in the spectrum of $T$ is necessarily a simple pole, and thus equal to some $\lambda_n$. 

Recall that if  $\mu$ is an isolated  point of the spectrum, for the resolvent $R(\lambda, T)$ we have the Laurent development 
\[ R(\lambda, T) =\sum_{k=-\infty}^{\infty} A_k (\lambda-\mu)^k, \]
which is valid for $0<|\lambda-\mu|<\delta$ and   $\delta= \dist(\mu,\sigma(T)\setminus\{\mu\})$. Here  $A_k\in\LX$ are the coefficients and $A_{-1}=P$ is the spectral projection associated with $\mu$. Thus the spectral projection is equal to the residuum. 

One says that $\mu\in\sigma(T)$ is a \emph{simple pole} if it is isolated in $\sigma(T)$ and if $\dim(\rg P)=1$. This implies that $A_k=0$ for $k\leq -2$. Moreover $\rg P=\ker(\mu\Id-T)$ and $PT=TP=\mu P$. 

\begin{thm}\label{th:42}
	Let $\mu\in\C\setminus \{0\}$ be an isolated point of $\sigma(T)$. Then there exists $n\in\N_0$ such that $\mu=\lambda_n$ and  $\mu$ is a simple pole. Moreover $P_nX\subset X$  and ${P_n}_{|X}$ is  the spectral projection associated with $\mu$. Here $P_n$ is the projection from Theorem~\ref{th:31}.  
\end{thm}
\begin{proof}
Assume that $\mu\not\in \{ \lambda_n:n\in\N_0\}$, $\lambda_n=\varphi'(\alpha)^n$. Let $\varepsilon>0$ such that $D(\mu, 2\varepsilon)\cap \{\lambda_n:n\in\N_0\}=\emptyset$.  Denote by 
\[P=\frac{1}{2i\pi} \int_{|\lambda-\mu|=\varepsilon}(\lambda \Id-T)^{-1}d\lambda\]
the spectral projection associated with $\mu$. Let $f\in X, z\in\D$.
Since for $|\lambda-\mu|=\varepsilon$, $\lambda\in\rho(C_\varphi)\cap\rho(T)$, one has: 
\[ ((\lambda\Id -T)^{-1} f)(z)=((\lambda\Id-C_\varphi)^{-1}f)(z),    \]
 it follows from Lemma~\ref{lem:39} and Cauchy's theorem that $(Pf)(z)=0$. Since $f\in X, z\in\D$ are arbitrary, it follows that $P=0$, a contradiction. 

Thus $\mu=\lambda_{n_0}$ for some $n_0\in\N$. Let $\varepsilon>0$ such that $\lambda_n\not\in D(\lambda_{n_0},2\varepsilon)$ for all $n\neq n_0$. Let 
\[P=  \frac{1}{2i\pi} \int_{|\lambda-\lambda_{n_0}|=\varepsilon}(\lambda \Id-T)^{-1}d\lambda\]
be the spectral projection. It follows from Lemma~\ref{lem:310} that $P=P_{n_0}$. Thus $P$ has rank $1$ and this means by definition that $\mu=\lambda_{n_0}$ is a simple pole.  
\end{proof}	
\begin{rem}
	In \cite{delhi1} it was proved that each pole is necessarily simple. Now we know more: each \emph{isolated point in the spectrum is a simple pole}. 
\end{rem}
%
We note the following consequence of Theorem~\ref{th:42} 
\begin{cor}
	 If the spectrum of $T={C_\varphi}_{|X}$ is countable, then $\sigma(T)\subset\{\lambda_n:n\in \N_0\}\cup \{0\}$. 
\end{cor}
 \begin{proof}
Let $\mathcal U$ be an open neighborhood of $\{\lambda_n:n\in N_0\}\cup \{0\}$ and $K:=\sigma(T)\setminus {\mathcal U}$. Then $K$ is compact and countable. If $K\neq\emptyset$, then, by Baire's theorem, $K$ has an isolated point. This is impossible by Theorem~\ref{th:42}. Thus $K=\emptyset$. Since $\mathcal U$ is arbitrary, the claim follows.

\end{proof}

Our next aim is to describe the connected component of $0$ in $\sigma(T)$. 
 Assume that $X\hookrightarrow\Hol(\D)$ is invariant under $C_\varphi$ and let $T={C_\varphi}_{|X}$, as before, where $\varphi:\D\to\D$ is the given Schr\"oder map. Let us assume that  $\dim X=\infty$. Then $0\in\sigma(T)$ and the set 
 \[ J:=\{ n\in\N_0:{\Psi_n}_{|X}\neq 0\}  \]
is infinite   by Proposition~\ref{prop:64}
 
 We let $\lambda_n=\varphi'(\alpha)^n$ where $\alpha$ is the interior fixed point of $\varphi$. 
 By Proposition~\ref{th:41}, $\lambda_n\in \sigma(T)$ for $n\in J$. 
 Moreover, let 
 \[  J_0:=\{  n\in\N_0: \lambda_n\mbox{ is an isolated point of }\sigma(T)\}.   \]  
 We know that for $n\in J_0$, $\kappa^n\in X$ and $T\kappa^n =\lambda_n \kappa^n$. 
 
 Our main result in this section is the following quite precise 
 description of the spectrum of $T$. 
 \begin{thm}\label{th:lenoir}
 	Assume that $\dim X=\infty$. 
 	Denote by $\sigma_0(T)$ the connected component of $0$ in $\sigma(T)$. Then 
 	\[  \sigma(T)=\sigma_0 (T) \cup\{ \lambda_n:n\in J_0 \}.  \] 
 	In particular, $\sigma_0(T)$ is the essential spectrum of $T$.  
 \end{thm}
Of course it may happen that $J_0=\emptyset$. This is the case if and only if $\sigma(T)$ is connected. 

For the proof, we need the following. 
\begin{lem}\label{lem:lenoir}
Let $\sigma_1$ be an open and closed subset of $\sigma(T)$. \\
  If $0\not\in\sigma_1$, then there exists a finite set  $J_1\subset J_0$ such that 
\[\sigma_1 = \{\lambda_n:n\in J_1\}.\]  
\end{lem}  
\begin{proof}
  Assume that $0\not\in\sigma_1$. 	Let $\Gamma$ be a rectifiable Jordan curve such that $\sigma_1\subset \mathop{int}\Gamma$ and $\sigma(T)\setminus \sigma_1\subset \mathop{ext}\Gamma$. We can choose $\Gamma$ such that $\lambda_n\not\in \Gamma$ for all $n\in\N_0$, and $0\in\mathop{ext}\Gamma$. 
 Consequently $J_2:=\{n\in\N_0:\lambda_n\in\mathop{int}\Gamma\}$ is finite. Denote by 
 \[    P=\frac{1}{2i\pi}\int_{ \Gamma}R(\lambda, T)d\lambda\]
 the spectral projection with respect to $\sigma_1$. Let $Y:=PX$. Then $TY\subset Y$ and $\sigma(T_{|Y})=\sigma_1$. Let $f\in X,z\in\D$. Then 
 \[  (R(\lambda,T)f)(z)=((\lambda\Id -C_\varphi)^{-1}f)(z)  \]
 for all $\lambda\in\Gamma$.  Now we choose $\varepsilon>0$ small enough and deduce from Lemma~\ref{lem:39} and \ref{lem:310} that 
 \begin{eqnarray*}
 	 (Pf)(z)  & =  &\frac{1}{2i\pi} \int_{ \Gamma} (R(\lambda,T)f)(z)d\lambda \\
 	  & = & \frac{1}{2i\pi} \int_{ \Gamma}((\lambda\Id -C_\varphi)^{-1}f)(z)d\lambda\\
 	  & = & \sum_{k\in J_2}\frac{1}{2i\pi}\int_{|\lambda-\lambda_k|=\varepsilon} ((\lambda\Id -C_\varphi)^{-1}f)(z)d\lambda\\
 	  & = & \sum_{k\in J_2} \langle \Psi_k,f\rangle \kappa^k (z)\\
 	  & = & \sum_{k\in J_1} \langle \Psi_k,f\rangle \kappa^k (z),\\
 	 \end{eqnarray*} 	
  where $J_1:=J_2\cap J$. 
  
  Since ${T'\Psi_k}_{|X}={\lambda_k\Psi_k}_{|X}$ for all $k\in J_1$, and since the $\lambda_k$ are all different, it follows that the $\Psi_k$, $k\in J_1$, are linearly independent. Consequently we find $f_l\in X$ such that 
  $\langle \Psi_k, f_l\rangle =\delta_{kl}$  for $k,l\in J_1$. It follows that the $\kappa^k$, $k\in J_1$, form a basis of $Y$ consisting of eigenvectors of $T$, $T\kappa^k =\lambda_k\kappa^k$, $k\in J_1$. Thus 
  \[  \sigma_1=\sigma(T_{|Y})=\{  \lambda_k:k\in J_1 \}.  \]
\end{proof}	   

\begin{proof}[Proof of Theorem~\ref{th:lenoir}]
Let $K=\sigma(T)\setminus \{  \lambda_n:n\in J_0\}$. Then $K$ is compact and $0\in K$. It suffices to show that $K$ is connected. 

Let $\sigma_1\subset K$ be open and closed. We have to show that $\sigma_1=\emptyset$ or $\sigma_1=K$. \\
\emph{First case}: $0\not\in\sigma_1$. \\
We show that $\sigma_1$ is open in $\sigma(T)$. Let $z_0\in\sigma_1$. Then $z_0\neq 0$ and $z_0\neq \lambda_n$ for all 
$n\in J_0$. Thus there exists $\varepsilon>0$ such that $z\neq \lambda_n$ for all $n\in J_0$ if $|z-z_0|<\varepsilon$ and, if $z\in K$, then $z\in\sigma_1$. Hence also $D(z_0,\varepsilon)\cap \sigma(T)\subset \sigma_1$.  This proves that $\sigma_1$ is open in $\sigma(T)$. It is trivially closed. By Lemma~\ref{lem:lenoir} there exists a finite set $J_1\subset J_0$ such that $\sigma_1\subset \{\lambda_n:n\in J_1\}$. 
Since $\sigma_1\subset K$, it follows that $\sigma_1= \emptyset$. \\
\emph{Second case}: $0\in\sigma_1$. \\
Then $K\setminus \sigma_1=\emptyset$ by the first case. Thus $\sigma_1=K$.\\
By Corollary~\ref{cor:55}, the point $0$ is not isolated in $\sigma(T)$. Thus $0\in\sigma_e(T)$. It follows that $\sigma_0(T)\subset\sigma_e(T)$. Since (by Theorem~\ref{th:42}) $\sigma(T)\setminus \sigma_0(T)$ consists of Riesz points, we conclude that $\sigma_0(T)=\sigma_e(T)$. 
 \end{proof}
 Of course, it can happen that $\sigma_0(T)=\{0\}$. Here is a situation where it is bigger.   
\begin{cor}
	Let $n_0\in\N_0$. Assume that $\lambda_{n_0}\in\sigma(T)$ but $\kappa^{n_0}\not\in X$. Then $\lambda_{n_0}\in\sigma_0(T)$. In particular $r_e(T)\geq |\lambda_{n_0}|$. 
\end{cor} 
 \begin{proof}
 If follows from Theorem~\ref{th:42} that $\lambda_{n_0}$ is not an isolated point of $\sigma(T)$. Thus $\lambda_{n_0}\not\in \{ \lambda_n:n\in J_0  \}$. Hence $\lambda_{n_0}\in\sigma_0(T)$.  	
 \end{proof}	 
  We cannot expect in the general situation we are considering here to prove more precise results on the geometry of $\sigma_0(T)$. However, for several concrete Banach spaces $X$, it is known that $\sigma_0(T)$ is a disc. Moreover one can estimate the essential radius $r_e(T):=\sup\{ |\lambda|:\lambda\in\sigma_e(T)  \}$ by knowing whether $\kappa^n\in X$ or $\kappa^n\not\in X$. We explain this in the following examples.  
  
  \begin{exam}
  	   \begin{enumerate}
  	   	\item Let $X=H^2(\D)$ and let $\varphi:\D\to\D$ be a Schr\"oder map with fixed point $\alpha\in\D$ and Koenigs eigenfunction $\kappa$. Then $C_\varphi (X)\subset X$. Let $T={C_\varphi}_{|X}$. Then it is known 
  	   	that $\sigma_0(T)$ ($=\sigma_e(T)$) is a disc (see \cite[Theorem 7.30]{CM}). A question which has been investigated is for which $n$ the eigenfunction $\kappa^n$ of $C_\varphi$ lies in $X$, which means, for which $n$ actually $\lambda_n=\varphi'(\alpha)^n\in\sigma_p(T)$. This is related to the essential spectral radius $r_e(T)$.   
  	   	\begin{enumerate}
  	   		\item[(a)] One has $\sigma_0(T)=\{0\}$ if and only if $\kappa^n\in X$ for all $n\in\N_0$ 
  	   		(see \cite[Section 6]{shap98}). To say that $\sigma_0(T)=\{0\}$ is the same as saying that $T$ is a Riesz operator. We had seen in Section~\ref{sec:4} that $C_\varphi$ is always a Riesz operator on $\Hol(\D)$. So the situation is very different if we restrict $C_\varphi$ to a Banach space. 
  	   		\item[(b)] Let $n\in\N$. Then $\kappa^n\in X$ if and only if $|\lambda_n|>r_e(T)$. 
  	   		\begin{proof}
  	   		If $|\lambda_n|>r_e(T)$, then $\lambda_n\in\sigma_p(T)$. Thus $\kappa^n\in X$ by Koenigs' theorem (Theorem~\ref{th:koenigs}).  The converse implication follows from deep results by Poggi-Corradini \cite{PC} and Bourdon--Shapiro \cite{BS}. We follow the survey article \cite{shap98} by Shapiro. Assume that $\kappa^n \in X$. Then $\kappa\in H^{2n}(\D)$. Using the notation of \cite[Section 7]{shap98}, this implies that $h(\kappa)\geq 2n$. By \cite[(12) in Section 6]{shap98}, one has $r_e(T)=|\lambda_1|^{h(\kappa)/2}$. Since $|\lambda_1|<1$ this implies $r_e(T)\leq |\lambda_1|^n=|\lambda_n|$. We need the strict inequality. Assume that $|\lambda_n|=r_e(T)$. Then $h(\kappa)=2n$. By the "critical exponent result" in \cite[Section 8]{shap98}, this implies that $\kappa^n\in X$. thus $|\lambda_n|>r_e(T)$.    	
  	   	    \end{proof}	
  	   		\item[(c)] The result of (b) can be reformulated by saying that there are no "hidden eigenvalues" besides possibly $\lambda_0$. More precisely, if $\lambda_n$ is an eigenvalue, then $\lambda_n\not\in \sigma_e(T)$. For $\lambda_0=1$ the situation is different. In Example~\ref{ex:1} an inner function 
  	   		$\varphi$ is defined  which is Schr\"oder and $0$ at $0$. Thus $T={C_\varphi}_{|H^2}$ is isometric and non invertible (see \cite[Section 1]{delhi2} for further informations on isometric composition operators on various Banach spaces). Hence $\sigma(T)=\overline{\D}=\sigma_e(T)$, and thus the eigenvalue $\lambda_0=1$ is in the essential spectrum.        
  	   	\end{enumerate}
  	   	\item Also on some weighted Hardy spaces the essential spectrum is a disc. However it can happen that 
  	   	$\lambda_{n+1}<r_e(T)<\lambda_n$ for some $n$. In fact Hurst \cite{hurst} considered Schr\"oder   
  	   	symbols $\varphi$ which are linear fractional maps with a fixed point $\alpha\in\D$ such that $\varphi'(\alpha)\in(0,1)$ and a second  fixed point of modulus one. The Banach spaces on which the composition operators are defined are the  weighted Hardy spaces 
  	   	\[H^2(\beta):=\left\{f(z)=\sum_{n\geq 0}^\infty a_nz^n:\|f\|^2=\sum_{n\geq 0}^\infty |a_n|^2\beta(n)^2<\infty  \right\}\]
  	   	where $\beta(n)=(n+1)^a$, $a<0$.  Recall that for $a=-1/2$, the Banach space is the classical Bergman space $\mathcal B$. 
  	   	In this case the spectrum is 
  	   	\[   \{\lambda\in\C: |\lambda|\leq \varphi'(\alpha)^{(2|a|+1)/2}\} \cup \{\varphi'(\alpha)^n:n\in \N_0\},\]
  	   		and $\kappa^p\in H^2(\beta)$ if and only if $p<|a|+1/2$. For $a=-1$, it follows that $\lambda_2<r_e(T)<\lambda_1$, whereas, for $a=-1/2$, we get $r_e(T)= \lambda_1$. 
  	   \end{enumerate}
   \end{exam}
   
 \begin{rem}
 By Proposition~\ref{th:41} we have seen that the composition operators associated with a Schr\"oder symbol $\varphi$ are not quasinilpotent on a large class of Banach spaces of holomorphic functions. Note that if $\varphi$ has a  fixed point $\alpha \in \D$ such that $\varphi'(\alpha)=0$, the description of the spectrum of $T:={C_\varphi}_{|X}$ may be very different. For example, if $\varphi(z)=z^2$ and $X=zH^2(\D)$, then the spectrum of $T$ is the closed unit disc ($T$ is a non-invertible isometry)  whereas for $X=z{\mathcal B}$, $T$ is quasinilpotent (see \cite[Theorem 2.9]{delhi2}).         	
 \end{rem}	

\textbf{Acknowledgments:}  This research is partly supported by the B\'ezout Labex, funded by ANR, reference ANR-10-LABX-58.  The authors are also grateful to R. Lenoir for stimulating discussions.    
 
  \bibliographystyle{amsplain}
   	
\end{document}